\documentclass[12pt]{article}

\usepackage[reqno]{amsmath}
\usepackage{amssymb,amsthm}
\usepackage{enumerate}

\allowdisplaybreaks

\newtheorem{theorem}{Theorem}
\newtheorem{lemma}{Lemma}
\newtheorem{conjecture}{Conjecture}
\newtheorem{corollary}{Corollary}

\newcommand{\ph}{\varphi}
\newcommand{\eps}{\varepsilon}
\newcommand{\hb}{1/2}

\newcommand{\zeile}{\vspace{\baselineskip}}

\newcommand{\N}{\mathbb{N}}

\newcommand{\R}{\mathbb{R}}
\newcommand{\C}{\mathbb{C}}

\zeile{\vspace{\baselineskip}}

\begin{document}
\markboth{Karin Halupczok}%
{Goldbach's problem with primes in APs
and in short intervals}

\title{Goldbach's problem with primes in arithmetic progressions
and in short intervals} 

\author{Karin Halupczok}

\maketitle

\begin{abstract}
Some mean value theorems in the style of Bombieri-Vinogradov's theorem
are discussed. They concern binary and ternary additive problems
with primes in arithmetic progressions and short intervals.
Nontrivial estimates for some of these mean values are given.
As application inter alia, we show that for large odd $n\not\equiv 1\:(6)$,
Goldbach's ternary problem 
$n=p_1+p_2+p_3$ is solvable with primes $p_{1},p_{2}$ in short intervals
$p_i \in [X_i,X_i+Y]$ with $X_{i}^{\theta_{i}}=Y$, $i=1,2$, and
$\theta_{1},\theta_{2}\geq 0.933$ such that 
$(p_{1}+2)(p_{2}+2)$ has at most $9$ prime factors. 
\end{abstract}


\zeile\zeile
\textbf{Notations:} 
By $p,p_1,p_2,p_3$ we denote prime numbers. 
The symbol $X\asymp Y$ means $X\ll Y\ll X$, and the
symbol $n\sim N$ denotes the range $N\leq n<2N$ for $n$.
We write $a\:(q)$ for a residue class $a\mod q$. 
Throughout, a star at a residue sum or maximum means that the
sum or maximum goes over all reduced residues.
By $\tau(q)$ we denote the number of positive divisors of $q$,
and by $\nu(q)$ the number of prime factors of $q$.
The symbol $P_{s}$ stands for a pseudo-prime of type $s$, that is a
positive integer with at most $s$ prime factors.
Further, $\varepsilon$, $\varepsilon_{1}$ and $\varepsilon_{2}$
are small positive real constants.
By $A>0$ we denote a given positive constant, and $B=B(A)>0$ denotes
a positive constant depending only on $A$. All implicit constants
may depend on $A$ and $\eps$, $\eps_{1}$, $\eps_{2}$.

By $Q_{1}$ and $Q_{2}$ we denote real numbers $\geq 1$ serving as
bounds for the moduli of the considered arithmetic progressions.

For the least common multiple of two integers $a$ and $b$,
we write $[a;b]$, and $(a,b)$ denotes their greatest common divisor.

Considering intervals, we use the notation
$[X, X+Y]$ for the set of integers $n$ with $X<n\leq X+Y$,
where $X$ and $Y$ denote real numbers $\geq 2$.
We call such an interval short, if $Y=X^{\theta}$
for a real number $0<\theta<1$. 

The real numbers $R,Y,X,Y_{1},Y_{2},X_{1},X_{2}\geq 2$ are considered to be 
sufficiently large, being at least as big as some constant depending
only on $A$ and the $\varepsilon$, $\varepsilon_{1}$, $\varepsilon_{2}$.
Further $L:=\log Y$.

The statements of the Theorems 1,2,3,4,6 in this article
begin with ``For all $A>0$ there is a $B=B(A)>0$ such that for all
$R,Y,X,X_{1},\dots$ (list of occurring parameters) the following holds:''.
Theorems 5,7 begin with
``For all
$A>0$ and all $R,Y,X,X_{1},\dots$ (list of occurring parameters) 
the following holds:''
This sentence is left out for an easier reading.

\section{Introduction}
\label{intro}

\subsection{Statement of results}

This article examines binary and ternary additive problems with primes in 
arithmetic progressions (APs for short) and in short intervals. We study
what kind of mixtures of such conditions on the prime summands
are treatable and where the limits of current methods are,
especially when treating ternary problems. 
We show two such ternary theorems resulting from two different
approaches and give corollaries for additive problems with 
almost-twin primes in short intervals.
Here we call a prime $p$ almost-twin, if $p+2$ is an
almost-prime $P_{s}$ for some positive integer $s$.

The first approach works with an estimate that goes back
to Kawada in \cite{Kaw} and leads to the following Theorem
\ref{th:1}, in which we consider the ternary Goldbach problem with
two primes in APs, both lying in short intervals
being of the same length $Y$:

\begin{theorem}
  \label{th:1}
Let $n\geq X_1+X_2+2Y$ be odd, let $n\ll X_{1}Y$,
let $X_{2}\geq Y\gg (n-X_{1})^{2/3+\varepsilon_1}$, let $X_{1}\geq 
Y\gg X_1^{3/5+\varepsilon_2}$ 
and assume that $Q_{i}\ll YX_{i}^{-1/2}L^{-B}$ for $i=1,2$.

Then for any fixed integers $a_{1},a_{2}$ with $a_{1}\leq
n-X_{1}-Y$ we have
\begin{multline*}
      \sum_{q_1\leq Q_1}
      \sum_{q_2\leq Q_2}
      \Big| \sum_{\substack{p_1+p_2+p_3=n\\p_i \in [X_i,X_i+Y],
          \;i=1,2\\ p_i\equiv a_i\: (q_i), \;i=1,2}}
     \log p_1 \log p_2 \log p_3 - 
\mathfrak{T}(n, q_1,a_1,q_2,a_2) Y^{2} \Big| \\  \ll Y^2L^{-A}.
\end{multline*}
\end{theorem}

The singular series $\mathfrak{T}(n,q_1,a_1,q_2,a_2)$ contains the
whole arithmetic information of the problem.
It is given at the end of Section \ref{seczweiAP} in its Euler
product form.

Using a sieve theorem for a two-dimensional sieve,
we deduce from this result the following ternary corollary.

\begin{corollary}
 \label{cor:1}
  Let $Y$ be large and consider any odd integer $n\not\equiv 1\:(6)$ 
  with $n\geq X_1+X_2+2Y$ and $n\ll X_1 Y$. Then
  the equation $n=p_1+p_2+p_3$ is solvable in primes $p_1,p_2,p_3$
  such that $(p_1+2)(p_2+2)=P_9$, 
  where $p_i \in [X_i,X_i+Y]$,  $X_1^{\theta_{1}}=Y$, 
  $(n-X_1)^{\theta_2}=Y$ with $\theta_{i}\geq 0.933$, $i=1,2$.  
\end{corollary}

Another variant of this corollary can be deduced as follows.

As a first step for this, we deduce in Section \ref{sth}
a short interval version of Meng's result in \cite{Meng}:

\begin{corollary}
  \label{cor:2}
For all but ${\rm O}(YL^{-A})$ even integers $2k_1\not\equiv 2\:(6)$ with
$k_1\in [X_{1},X_1+Y]$, the equation $2k_1=p_2+p_3$  is solvable in primes
$p_2,p_3$ such that $p_2+2=P_3$, $p_2\in [X_{2},X_2+Y]$, 
where $X_2^\theta = Y$ with $\theta\geq 0.861$, and $X_2+Y\leq
2X_1\ll Y^{3/2-\eps}$. 
\end{corollary}

By a counting argument, we infer in Section \ref{sth}
a variant of Corollary \ref{cor:1}, using
Corollary \ref{cor:2} and a theorem of Wu in
\cite{Wu2} on the number of Chen primes in short intervals.

\begin{corollary}
  \label{cor:3}
  Let $X_1^{\theta_1}=Y=X_2^{\theta_2}$ be large, where $\theta_1\geq 0.971$
  and $\theta_2\geq 0.861$. Let $n$ be an odd integer $n\not\equiv 1\:(6)$
  with $X_1+X_2+2Y\leq n \ll Y^{3/2-\eps}$. Then
  the equation $n=p_1+p_2+p_3$ is solvable 
  in primes $p_1,p_2,p_3$ such that $p_1+2=P_2$, 
  $p_2+2=P_3$ and $p_i \in [X_i,X_i+Y]$, $i=1,2$.
\end{corollary}

This corollary cannot be deduced from Corollary \ref{cor:1}
before since the almost-prime conditions on $p_{1}+2$ and 
$p_{2}+2$ in Corollary \ref{cor:3} are stronger. And also not
vice versa since the short interval condition in Corollary
\ref{cor:1} is stronger due to $0.971>0.933$.

\zeile
Now we state the results of the second approach, it leads 
to theorems of a similar kind. In this approach, we use
an adaption of the theorem of Perelli and Pintz in \cite{PePi}.

For the ternary Goldbach problem with one prime in an arithmetic progression
and two primes in given short intervals of different length, we show
in Section~\ref{seceinAP} the following result.

\begin{theorem}
\label{th:2}
Let $n$ denote a large positive odd integer, let
$X_{1}\geq Y_{1}\gg X_{1}^{3/5+\varepsilon_1}$, let $X_{2}\geq Y_{2}\gg
X_{2}^{7/12+\varepsilon_{2}}$. Assume that $Y_{2}^{1/3+\eps}\ll Y_{1}\ll Y_{2} 
\asymp n-X_1-Y_{1}-X_2-Y_{2}\geq 0$ holds.
Then, for $Q\ll Y_{1}X_{1}^{-1/2}L^{-B}$
with $Q\ll Y_{1}^{3/2}(n-X_{1})^{-1/2}$ and any fixed integer $a$ 
with $a\leq n-X_{1}-Y_1$ we have
\[
   \sum_{q\leq Q}
\Big| \sum_{\substack{p_1+p_2+p_3=n\\p_i \in [X_i,X_i+Y_i],\; i=1,2\\ p_1\equiv a\;(q)}}
     \log p_1 \log p_2 \log p_3 - 
\mathfrak{T}(n,q,a)Y_{1}Y_{2}
\Big| \ll Y_{1}Y_{2}L^{-A}.
\]
Further, if $Q\ll Y_1^{1/2}$ is assumed instead of $Q\ll
Y_{1}^{3/2}(n-X_{1})^{-1/2}$, 
then $\sideset{}{^{\displaystyle *}} \max_{a(q)}$ can be
inserted after the sum over $q$. 
\end{theorem}

The condition $n-X_1-Y_{1}-X_2-Y_{2}\asymp Y_{2}$ is rather restrictive:
If $n,X_1,Y_{1}$ are given,
one has to choose $X_2,Y_{2}$ then appropriately, but still 
feasible in such a way that
$Y_{1}$ and $Y_{2}$ may be of different magnitude. 
The theorem gives a sharp estimate then.

Here $\mathfrak{T}(n,q,a)$ denotes the ternary singular series with
one prime in an arithmetic progression, namely
\[
    \mathfrak{T}(n,q,a) := \frac{1}{\ph(q)}
    \prod_{\substack{p\mid n, p\nmid q \\\text{or } p\nmid n-a,p\mid q}}
    \Big(1-\frac{1}{(p-1)^2}\Big) \prod_{p\nmid nq} \Big( 1+\frac{1}{(p-1)^3} \Big)
    \prod_{p\mid n-a, p\mid q} \frac{p}{p-1}.
\]

From Theorem \ref{th:2} we deduce:
Every large odd $n$ can be written as $n=p_1+p_2+p_3$ with two primes in
short intervals of different length, one of which lying in an
arithmetic progression $a$ modulo $q$ for almost all admissible 
moduli $q\leq Q$, where $Q\ll Y_{1}X_{1}^{-1/2}L^{-B}$
and $Q\ll Y_{1}^{3/2}(n-X_{1})^{-1/2}$.

A last corollary with almost-twin primes in short intervals
of different interval lengths can be deduced from Theorem
\ref{th:2} above, again using sieve methods:

\begin{corollary}
  \label{cor:4}
 Let $X_1^{\theta_1}=Y_{1}$ be large, 
 let $X_{2}^{\theta_{2}}=Y_{2}$ with $\theta_2> 3/5$, and let 
 $Y_{2}^{\eta}\ll Y_{1} \ll Y_{2}$.
  Let $n$ be an odd integer $n\not\equiv 1\:(6)$
  with $Y_{2}\asymp n-X_1-X_2-Y_{1}-Y_{2}>0$. Then
  the equation $n=p_1+p_2+p_3$ is solvable in primes
  $p_{1},p_{2},p_{3}$ such that $p_{1}+2=P_{3}$ and 
  $p_{i}\in [X_{i},X_{i}+Y_{i}]$, $i=1,2$, if 
  \vspace{-1ex}
  \begin{itemize}
  \itemsep0pt
  \item either $1\geq\theta_{1}\geq 0.861$, 
    $1\geq\eta>(1+1/\theta_{1})^{-1}> 0.462$
  \item or $0.5\geq \eta\geq 0.463$,
    $1\geq\theta_{1}>1/\min(1/\eta-1,(3-1/\eta)\Lambda_{3}/2)> 0.782$,
    where $\Lambda_{3}:=4-\log(27/7)/\log 3$.
  \end{itemize}
\end{corollary}

\subsection{A conjectured unification of the results}

Now we ask what would be the strongest version of a theorem
that combines both classes of results.
All theorems above are deduced either by the Kawada-approach
or by the Perelli-Pintz-approach. It would be interesting if
there exists a slightly stronger theorem that would incorporate all
such results. Such a unification, which seems to be unreachable by
current methods, can be stated 
in the following way: 

\begin{conjecture}
There exist absolute constants $0<\theta,\theta_{1},\theta_{2}<1$ such that for
$X_{1}\geq R\gg X_{1}^{\theta_{1}}$, $X_{2}\geq Y\gg
X_{2}^{\theta_{2}}$ and $R\ll Y^{\theta}$,
the estimate
\[
  \sum_{\substack{r\in [X_{1},X_1+R]\\2\mid r}} \sum_{q\leq Q} \:
   \sideset{}{^{\displaystyle *}}\max_{\substack{a \:(q) \\ (a-r,q)=1}}
   \Big| \sum_{\substack{p_2\in [X_{2},X_2+Y] \\ p_2\equiv a\;(q) \\ p_2+p_3=r }} 
    \log p_2\log p_3 - \mathfrak{S}(r,q,a)Y \Big|\ll \frac{RY}{L^A}
\]
holds for all $Q\ll YX_{2}^{-\hb} L^{-B}$.
Here the numbers $r$ and $p_{2}$ are chosen from
short intervals of length $R$ respectively $Y$. 
In this statement, the Goldbach equation $p_2+p_3=r$ may be replaced
by the twin equation $p_2-p_3=r$.
\end{conjecture}
We note that a version of
the conjecture with $Y\ll R^{\theta}$ would also be desirable, 
but this seems to be also hard.

Here the singular series $\mathfrak{S}(r,q,a)$ is the expression
\begin{equation}
\label{eq:sing1}
   \mathfrak{S}(r,q,a):=\begin{cases} \frac{1}{\varphi(q)}
     \mathfrak{S}(rq), & \text{if } 2\mid r, \ (a,q)=(a-r,q)=1,\\
    0, &\text{otherwise}, \end{cases}
\end{equation}
with
\begin{equation}
\label{eq:sing2}
   \mathfrak{S}(r):=\begin{cases}2\prod_{p\neq 2}
     (1-\frac{1}{(p-1)^{2}})
     \prod_{\substack{p \mid r\\p\neq 2}} \frac{p-1}{p-2},
     &\text{if }2\mid r,\\
    0, &\text{otherwise}.\end{cases}
\end{equation}
Given as a series, it can be written as
\[
  \mathfrak{S}(r,q,a)= \sum_{s\geq 1} H_{s}(r,q,a)
\] with
\begin{equation}
  \label{eq:sing3}   
H_{s}(r,q,a):=
  \frac{\mu(s)}{\varphi(s)\varphi([q;s])} \sideset{}{^{\displaystyle *}}\sum_{b(s)}
  e\Big(\frac{-rb}{s}\Big) \sideset{}{^{\displaystyle
      *}}\sum_{\substack{c(s)\\(q,s)\mid c-a}} 
  e\Big(\frac{bc}{s}\Big),
\end{equation}
see e.g.\ \cite[eq.~(33)]{BCD}. Further, we have
\begin{equation}
\label{eq:sing4}
\mathfrak{S}(r)=\sum_{s\geq 1}  \frac{\mu^2(s)}{\varphi^2(s)}
\sideset{}{^{\displaystyle *}}\sum_{b(s)}e\Big(\frac{-rb}{s}\Big).
\end{equation}

A number of theorems proved by several authors can be seen as 
special cases of this conjecture. Mikawa \cite{Mik} considered
the case for $R\asymp Y\asymp X_{1}\asymp X_{2}$, and 
Laporta \cite{Lap} the one for $Y\asymp R$ and short intervals 
for $r$. 
Meng \cite{Meng},\cite{Meng2}
considered the Goldbach variant for $R\asymp Y$
and non-short intervals. Kawada's estimate \cite{Kaw} for the
special case $k=2$, $a_0=1$, $b_0=0$, $a_1=\pm 1$, $b_1=r$
is contained in the conjecture with $R\asymp Y$, 
with short intervals for both $p_2$ and $r$.
Perelli's and Pintz's result \cite{PePi} is contained in the conjecture 
when taking no primes in arithmetic progressions, but such that $r$ 
lies in a short interval.

By now, the paper \cite{BCD} of A.\ Balog, A.\ Cojocaru and C.\ David
contains an interesting result of this kind
which can be seen closest to the conjecture, namely, that
the Barban/Davenport/Halberstam-variant of the conjecture is 
true, that means, the
corresponding estimate when max over $a$ is replaced by sum over $a$. 
It suggests that also the 
conjecture, which is a sharper estimate of Bombieri-Vinogradov-type,
could be true. It would be the next step for reaching
stronger results in this area.

We remark that the major arc contribution of the conjecture
can be shown with known standard methods. 
We obtain that $R\ll Y$ with $Y\gg X_{1}^{7/12+\eps_{1}}$, $Y\gg
X_{2}^{3/5+\eps_{2}}$, $X_{1}-X_{2}\geq Y$ suffices.  
The problem to prove the conjecture lies in the minor arc contribution.

Theorems \ref{th:2} and \ref{th:1}
show that we can put one AP-condition
and two short-interval-conditions or two AP-conditions and
one short-interval-condition (two with the same length are counted
as one) on the primes when treating the ternary
Goldbach problem. But stronger versions seem to be hard.
The stated conjecture would lead to a version with two AP-conditions
and two short-interval-conditions.

We would like to mention that a mean value theorem 
as Theorem \ref{th:1}
could also hold with all three primes
in arithmetic progressions with large moduli.
But this problem is of similar kind and seems to be unreachable, 
too. See also \cite{Tol}, \cite{Tol2},\cite{Tol3} and \cite{HalAMH} 
for discussions of this conjecture.

\section{Tools}
\label{HalMont}

\zeile 
As a corollary of the known inequalities of Halasz and Montgomery, 
see Satz 7.3.1 \cite{Brue}, we first state the following lemma.
\begin{lemma}
  \label{lHM1} Consider a fixed integer $a$.
 For a real bound $Q\geq 1$, positive integers $M,N$ with $Q\leq M$, $N\leq M$,
 $M\geq a$ and $v_{M+1},\dots,v_{M+N}\in\C$ we have
\begin{equation}
\label{eq:lHM1}
   \sum_{q\sim Q} \Big|\sum_{\substack{n\in [M,M+N] \\n\equiv a\;(q)}}
   v_{n}\Big| \ll\; (N+Q^{2/3}M^{1/3})^{1/2}(\log (M+1))^{3/2} \Big(\sum_{n\in [M,M+N]}
   |v_{n}|^{2}\Big)^{1/2}.
\end{equation}
\end{lemma}

\begin{proof}
Write $v:=(v_{M+1},\dots,v_{M+N})^{T}\in\C^{N}$ and consider the usual
scalar product $\langle v,w\rangle:=\sum_{n\sim N}v_{n}\overline{w_{n}}$ on
$\mathbb{C}^{N}$, where $v,w\in\C^{N}$. For every $q\sim Q$ let
\[
   \ph_{q}(n):=\begin{cases} 1, &\text{ if } n\equiv a \;(q) \\
                             0, &\text{ else,} \end{cases}
\]
so that $\ph_{q}\in\C^{N}$.

The left hand side of \eqref{eq:lHM1} then becomes $\sum_{q\sim Q}|\langle
v,\ph_{q}\rangle|$. Halasz-Montgomery's inequality states that this is
\[
\leq \| v\| \Big( \sum_{q_{1},q_{2}\sim Q} |\langle
\ph_{q_{1}},\ph_{q_{2}}\rangle|\Big)^{1/2},
\]
where $\|v\|:=\Big( \sum_{n\in [M,M+N]} |v_{n}|^{2}\Big)^{1/2}$.

Write $\tau_{Q}(n):=\sum_{d\mid n,d\sim Q}1$.
Now we have, since $M\geq a$ (also $Q,M\geq 2$ may be assumed w.l.o.g),
\begin{align}
   &\sum_{q_{1},q_{2}\sim Q} |\langle \ph_{q_{1}},\ph_{q_{2}}\rangle| =
   \sum_{q_{1},q_{2}\sim Q} \sum_{n\in [M,M+N]}\ph_{q_{1}}(n)
   \ph_{q_{2}}(n)  = \sum_{n\in [M,M+N]} \sum_{\substack{q_{1}\sim Q \\ q_{1}\mid n-a}} 
      \sum_{\substack{q_2\sim Q\\ q_{2}\mid n-a}} 1 \notag \\
  &=  \sum_{s,0\leq M-a< s\leq M+N-a} \sum_{\substack{q_1\sim Q \\ q_{1}\mid s}} 
      \sum_{\substack{q_2\sim Q\\ q_{2}\mid s}} 1
    \leq \sum_{s\in [M-a,M+N-a]} \tau_{Q}(s)^{2} \notag \\
    &= \sum_{s\in [M-a,M+N-a]} \tau_{Q}(s) \sum_{s=uv,u\sim Q} 1
    = \sum_{uv\in [M-a,M+N-a], u\sim Q} \tau_{Q}(uv)\notag \\
    &\leq \sum_{uv\in [M-a,M+N-a], u\sim Q} \tau(u)\tau(v)
    =\sum_{u\sim Q} \tau(u) 
       \sum_{v\in[\frac{M-a}{u},\frac{M+N-a}{u}]}\tau(v). \label{eq:divsum}
\end{align}
For the inner sum, we use Vorono\"i's estimate
\[
    \sum_{n\leq x}\tau(n) = x\log x+(2\gamma-1)x+O(x^{1/3}\log x),
\]
and continue \eqref{eq:divsum} with
\begin{align*}
  &\ll \sum_{u\sim Q} \tau(u) \Big(\frac{N}{u}\log M +
  \frac{M^{1/3}}{u^{1/3}}\log M\Big) \\
  &\ll N\log M\sum_{u\sim Q} \frac{\tau(u)}{u}
  + \sum_{u\sim Q} \frac{\tau(u)}{u^{1/3}} M^{1/3}\log M\\
  &\ll N(\log M)^{3} + Q^{2/3}M^{1/3}(\log M)^{3},
\end{align*}
where we used that $\sum_{u\leq Q}\frac{\tau(u)}{u}\ll (\log
Q)^{2}\leq (\log M)^{2}$.
\end{proof}

In a similar way, we can also prove the following variant which is
uniform in the residues. 
We can expect nontrivial estimates from this for the range
$Q\ll N^{\hb}$ for $Q$:
\begin{lemma}
  \label{lHM2} For a real $Q\geq 1$, integers
  $M,N\geq 2$ and $v_{M+1},\dots,v_{M+N}\in\C$ we have
\begin{equation}
\label{eq:lHM2}
   \sum_{q\sim Q} \max_{a\:(q)}
   \Big|\sum_{\substack{n\in [M,M+N]\\n\equiv a\:(q)}}
   v_{n}\Big| \ll\; (N\log (Q+1) +Q^2)^{1/2} \Big(\sum_{n\in [M,M+N]}
   |v_{n}|^{2}\Big)^{1/2}.
\end{equation}
\end{lemma}

\begin{proof}
  For $a$ mod $q$ define
\[
   \ph_{q,a}(n):=\begin{cases} 1, &\text{ if } n\equiv a \;(q), \\
                             0, &\text{ else,} \end{cases}
\]
and let $a=a_q$ be a residue mod $q$ such that $|\langle
v,\ph_{q,a}\rangle|$ is maximal. 
Then we have
\begin{align*}
   &\sum_{q_{1},q_{2}\sim Q} |\langle \ph_{q_{1},a_{q_1}},\ph_{q_{2},a_{q_2}}\rangle| =
   \sum_{q_{1},q_{2}\sim Q} \sum_{n\in [M,M+N]}\ph_{q_{1},a_{q_1}}(n)
   \ph_{q_{2},a_{q_2}}(n)  \\ &
  =   \sum_{q_1,q_2\sim Q} 
  \sum_{\substack{n\in [M,M+N] \\ n\equiv a_{q_1}\: (q_1) \\ n\equiv a_{q_2}\: (q_2) }} 1
   \leq \sum_{q_1,q_2\sim Q} \Big(\frac{N}{[q_1;q_2]} + 1 \Big)
  \\ &= Q^2 +  \sum_{q_1\sim Q} \frac{N}{q_1}\sum_{d\mid q_1} 
    \sum_{\substack{q_2\sim Q\\(q_2,q_1)=d}} \frac{d}{q_2} 
   \leq Q^2 + \frac{N}{Q} \sum_{q_1\sim Q} \sum_{d\mid q_1}
   \sum_{q_2'\sim Q/d} \frac{1}{q_2'} \\ 
&\ll Q^2  + \frac{N}{Q} \sum_{q_1\sim Q} \tau(q_1) \ll Q^2+ N \log Q. 
\end{align*}
Then again Halasz-Montgomery's inequality shows the assertion.
\end{proof}

As a further tool, we use
the theorem of Perelli, Pintz and Salerno in \cite{PePiSa}
in the following form, it is a Bombieri-Vinogradov theorem for short
intervals. 
\begin{theorem} 
\label{pepisa}
Let $X\geq Y\gg X^{3/5+\eps}$ and $Q\ll YX^{-1/2}(\log X)^{-B}$. Then
\[
   \sum_{q\leq Q}\max_{(a,q)=1} \Big|
      \sum_{\substack{p\in[X,X+Y]\\p\equiv a \:(q)}} \log p -
{\frac{Y}{\varphi(q)}}\Big|\ll  {\frac{Y}{L^{A}}}.
\]
\end{theorem}

Further we use a special case of Kawada's Theorem in \cite{Kaw}
in the following formulation:
\begin{theorem} 
   \label{th:kaw}  
Let $X_1^{2/3+\eps}\ll Y\leq X_{1}$, let $Y\leq X_{2}$, 
let $X_2\leq 2X_{1}-Y$ and $Q_{2}\ll
YX_{2}^{-1/2}L^{-B}$. Then for any integer $a_{2}$, we have
\begin{equation}
\label{eq:eqfor5} 
  \sum_{q_2\leq Q_2}
  \sum_{k_1\in [X_{1},X_1+Y]}
   \Big| \sum_{\substack{p_2+p_3=2k_1\\ p_2\equiv a_2\:(q_2) \\ p_2\in
       [X_{2}, X_2+Y]}}
   \log p_2 \log p_3 - \mathfrak{S}(2k_1,q_{2},a_{2})Y \Big| \ll Y^2 L^{-A},
\end{equation}
with $\mathfrak{S}$ as given in \eqref{eq:sing1}.
\end{theorem}

This is by Kawada's result \cite[Thm.~2]{Kaw} for the special case
 $k=2, a_0=1,b_0=0, a_1=-1,b_1=2k_1$, but with a small change
made, namely the left boundaries $X_1$ and $X_2$ of the short
intervals.
Originally, the restriction $X_{1}=X_{2}$ has been stated
in \cite{Kaw}, but Kawada's proof works also for any $X_{2}\leq
2X_{1}-Y$: In fact, for the minor arc contribution, the interval boundaries
in the exponential sums do not play a role 
due to the use of Bessel's inequality at
the end of \S 4 in \cite{Kaw}, cp.\ also
\cite[(4.3)--(4.5)]{Kaw},
where the case $k=2$ is treated individually.
For the minor arc contribution, just the assumption
$Y\gg X_{1}^{2/3+\varepsilon}$ is used.

The condition $X_{2}\leq 2X_{1}-Y$ is added to avoid that the necessary
interval for $p_{3}=2k_{1}-p_{2}$ has a negative left boundary and hence a
cut-off at $0$.
In Kawada's major arc treatment, the exponential sum range
for $p_{3}$ has then still length $\asymp Y$
(there the sum is termed  $P(\alpha)$ and is
approximated by $\frac{\mu(q)}{\varphi(q)}T(\beta)$).
Then the proof of (3.8) and (3.9) 
in \S 5 of \cite{Kaw} reads the same.
The major arc treatment uses then the assumption $Q_{2}\ll
YX_{2}^{-1/2}L^{-B}$ and $Y\gg X_{2}^{3/5+\varepsilon_{2}}$, but the
latter estimate follows from the other assumptions in Theorem
\ref{th:kaw}.

For the second approach, we use an adaption of the
Theorem of Perelli and Pintz in \cite{PePi}. It was independently found
by Mikawa in \cite{Mik2}.
Originally, this theorem states:
If $X_{1}^{1/3+\eps}\ll R\leq X_{1}$, then
\[
   \sum_{2k\in[X_{1},X_{1}+R]} \Big| \sum_{p_{2}+p_{3}=2k} \log p_{2}
   \log p_{3}  - \mathfrak{S}(2k)2k  \Big|^{2} \ll RX_{1}^{2} L^{-A}.
\]

The proof in \cite{PePi} can be adapted in such a way that
an additional short interval-condition can be put on one of 
the primes, namely it can be shown (cp.\ \cite[Thm.~3]{PePi}):
\begin{theorem}
  \label{th:GBshortIV}
   Let $Y^{1/3+\eps}\ll R \ll Y \asymp 2X_{1}-X_{2}-Y\geq 0$,
   let $R\leq X_{1}$
   and $X_{2}^{7/12+\eps_{2}}\ll Y \leq X_{2}$, then
   \[
        \sum_{k\in[X_{1},X_{1}+R]}  \Big|
        \sum_{\substack{p_{2}+p_{3}=2k\\ p_{2}\in [X_{2},X_{2}+Y]}}
        \log p_{2} \log p_{3} - \mathfrak{S}(2k)Y\Big|^{2} \ll RY^{2}L^{-A}.
   \]
\end{theorem}
The condition $Y \asymp 2X_{1}-X_{2}-Y$ comes from the necessary
interval condition for the prime $p_{3}$, in the original proof
in \cite{PePi} this plays an important role in the application
of Gallagher's Lemma and an argument due to Saffari and Vaughan
in \cite{SafVau}. It is not easy to delete this condition.
In contrast, the replacement of $2k$ by $Y$ in the main term
is an obvious adaption.

%

\section{Proof of Theorem  \ref{th:1}--Kawada-approach}
\label{seczweiAP}

We start by proving the following binary theorem
with one prime in a given arithmetic progression and lying in a
short interval.

\begin{theorem}
  \label{th:binGBtwoprimes}
Let $X_1^{2/3+\eps}\ll Y\leq X_{1}$, let $Y\leq X_2\leq 2X_{1}-Y$, 
let $Q_{1}\ll Y^{3/2}X_{1}^{-1/2}$ and $Q_{2}\ll YX_{2}^{-1/2}L^{-B}$.
Then, for any fixed integers $a_{1}$ and $a_{2}$, where
 $a_{1}\leq X_{1}$, we have
\begin{align*}
  \sum_{q_1\leq Q_1}
   \sum_{\substack{k_1\in[X_{1},X_1+Y]\\
      2k_1\equiv a_1\:(q_1)}}
   \sum_{q_2\leq Q_2}
    \Big| \sum_{\substack{p_2+p_3=2k_1 \\ p_2\equiv a_2 \:(q_2) \\
        p_2 \in [X_{2},X_2+Y]}}
    \log p_2 \log p_3 -\mathfrak{S}(2k_1,q_{2},a_{2})Y \Big| \\ 
       \ll Y^2 L^{-A}.
\end{align*}
If $Q_{1}\ll Y^{3/2}X_{1}^{-1/2}$ is replaced by $Q_{1}\ll Y^{1/2}$,
then the estimate holds true with $\max_{a_{1}\:(q_{1})}$ inserted
after $\sum_{q_{1}}$. 
\end{theorem}

Proof:
By Lemma \ref{lHM1}, for $a_{1}\leq X_1$ and since
$Q_{1}^{2/3}X_{1}^{1/3}\ll Y$, we deduce
from Kawada's Theorem \ref{th:kaw}:
\begin{align*}
  &\sum_{q_1\leq Q_1} \sum_{\substack{k_1\in [X_{1}, X_1+Y]\\ 2k_1\equiv a_1 \:(q_1)}}
  \sum_{q_2\leq Q_2} 
    \Big| \sum_{\substack{p_2+p_3=2k_1\\ p_2\equiv a_2\:(q_2)
      \\ p_2\in [X_{2}, X_2+Y]}}  \log p_2 \log p_3 
    - \mathfrak{S}(2k_1,q_{2},a_{2})Y \Big| \\
  \ll & Y^{1/2} L^{3/2} \Big(\sum_{k_1\in [X_{1},X_1+Y]}
 \Big(\sum_{q_2\leq Q_2}
 \Big| \sum_{\substack{p_2+p_3=2k_1\\ p_2\equiv a_2\:(q_2) \\ p_2\in [X_{2},X_2+Y]}}
   \log p_2 \log p_3 - \mathfrak{S} (2k_1,q_{2},a_{2})Y
      \Big| \Big)^{2} \Big)^{1/2} \\
  \ll & Y^{1/2} L^{3/2} \Big( YL^2 \sum_{k_1\in[X_{1},X_{1}+Y]}
   \sum_{q_2\leq Q_2}
 \Big| \sum_{\substack{p_2+p_3=2k_1\\ p_2\equiv a_2\:(q_2) \\
     p_2\in [X_{2}, X_2+Y]}}
   \log p_2 \log p_3 - \mathfrak{S} (2k_1,q_{2},a_{2})Y
      \Big| \Big)^{1/2} \\
  \ll & Y^{1/2} L^{3/2} (YL^2 Y^2 L^{-2A-5})^{1/2} \ll Y^2 L^{-A}
  \text{ by } \eqref{eq:eqfor5}.
\end{align*}

Note that for $Q_1 \ll Y^{1/2}$, we can apply Lemma \ref{lHM2}
in the same way 
to get an estimate which is uniform over the residues $a_{1}$; we get then
\begin{align*}
\sum_{q_1\leq Q_1} 
  \max_{a_1\:(q_1)}
  \sum_{\substack{k_1\in [X_{1},X_1+Y]\\ 2k_1\equiv a_1 \:(q_1) }}
  \sum_{q_2\leq Q_2}
\Big| \sum_{\substack{p_2+p_3=2k_1\\ p_2\equiv a_2\:(q_2) \\ p_2\in [X_{2}, X_2+Y]}}
   \log p_2 \log p_3 - \mathfrak{S} (2k_1,q_{2},a_{2})Y \Big| 
    \\ \ll Y^2 L^{-A}.
\end{align*}

So Theorem  \ref{th:binGBtwoprimes} follows. \qed


\zeile
Proof of Theorem \ref{th:1}: 

Write down the estimate of Theorem \ref{th:binGBtwoprimes},
but where the singular series in \eqref{eq:sing3} 
is replaced by the partial
sum for $s\leq L^C$. The estimate is still true since Theorem \ref{th:kaw}
is true in this form (see the treatment of $S_{2}$ in \cite[\S 6]{Kaw}).

Further, restrict the summation over $k_{1}$ to the $k_{1}$ of the
form $2k_1=n-p_1$ with $p_1\equiv a_1 \:(q_1)$, this gives then
\begin{multline*}
   \sum_{\substack{q_1\leq Q_{1} \\q_2\leq Q_{2}}}
   \sum_{\substack{p_1\equiv a_1\:(q_1)\\ p_1\in [X_{1}, X_{1}+Y]}}
 \Big| \sum_{\substack{p_2+p_3=n-p_{1}\\p_2\equiv a_2\:(q_2)\\ 
         p_2 \in [X_2,X_2+Y]}} \log p_2 \log p_3 
    -  \sum_{s\leq L^C} H_{s}(n-p_1,a_2,q_2)Y \Big|\\ \ll Y^2 L^{-A},
\end{multline*}
with $H_{s}$ as in \eqref{eq:sing3}.

Since we used Theorem \ref{th:binGBtwoprimes} for the $2k_{1}$-interval 
$[n-X_{1}-Y,n-X_{1}]$, the assumptions $n\geq X_{1}+X_{2}+2Y$, 
$Y\gg (n-X_{1})^{2/3+\varepsilon_{1}}$, $a_{1}\leq n-X_{1}-Y$,
$Q_{2}\ll YX_{2}^{-1/2}L^{-B}$ and $Q_{1}\ll Y^{3/2}(n-X_{1})^{-1/2}$ are
used (the latter one holds true since $n\ll X_{1}Y$, so 
$Q_{1}\ll Y^{3/2}X_{1}^{-1/2}Y^{-1/2}L^{-B}\ll Y^{3/2}(n-X_{1})^{-1/2}$).

In the previous estimate, we insert the weight $\log p_{1}$ and deduce that
\begin{align*}
   \sum_{\substack{q_1\leq Q_{1} \\ q_2\leq Q_{2}}}
   &\Big| \sum_{\substack{p_1+p_2+p_3=n\\p_i\equiv a_i\:(q_i), i=1,2\\ 
         p_i \in [X_i,X_i+Y_i], i=1,2}}
      \log p_1 \log p_2 \log p_3 \\
    &- \sum_{\substack{p_1\equiv a_1\:(q_1) \\ p_1\in [X_{1}, X_1+Y]}}
      \log p_1 \sum_{s \leq L^C} H_{s}(n-p_1,a_2,q_2)Y \Big| \ll Y^2 L^{-A}.
\end{align*}

Now the main term is
\begin{align*}
  &=Y \sum_{\substack{p_1\equiv a_1\:(q_1) \\ p_1\in [X_{1}, X_1+Y]}} \log p_1 \sum_{s\leq L^C}
   \frac{\mu(s)}{\ph(s)\ph([q_2;s])} \sideset{}{^{\displaystyle *}} \sum_{b\:(s)}
   e\,\Big( -(n-p_1)\frac{b}{s} \Big) 
    \sideset{}{^{\displaystyle *}} \sum_{\substack{c\:(s)\\ c\equiv a_2 ((q_2,s)) }} 
   e\,\Big(\frac{bc}{s} \Big)  \\
  &= Y  \sum_{\substack{p_1\equiv a_1\:(q_1) \\ p_1\in [X_1,X_1+Y]}} \log p_1 \sum_{s\leq L^C}
   \frac{\mu(s)}{\ph(s)\ph([q_2;s])} \sideset{}{^{\displaystyle *}} \sum_{b\:(s)}
      \sideset{}{^{\displaystyle *}} \sum_{\substack{c\:(s)\\ c\equiv a_2 ((q_2,s)) }} \: 
    \sideset{}{^{\displaystyle *}} 
    \sum_{\substack{d\:(s)\\ d\equiv a_1 ((q_1,s)) \\ d\equiv p_1 \:(s) }}
     e\,\Big( -(n-d-c)\frac{b}{s} \Big) \\
  &= Y \sum_{s\leq L^C} \frac{\mu(s)}{\ph(s)\ph([q_2;s])}
    \sideset{}{^{\displaystyle *}} \sum_{b\:(s)}
      \sideset{}{^{\displaystyle *}} \sum_{\substack{c\:(s)\\ c\equiv a_2 ((q_2,s)) }} \:
    \sideset{}{^{\displaystyle *}} \sum_{\substack{d\:(s)\\ d\equiv a_1 ((q_1,s)) }}
          e\,\Big( -(n-d-c)\frac{b}{s} \Big) 
    \sum_{\substack{ p_1\in [X_1,X_1+Y] \\ p_1\equiv a_1\:(q_1) \\ p_1\equiv d\:(s) }} \log p_1 \\
  &= Y^{2} \sum_{s\leq L^C}  
    \frac{\mu(s)G(n;a_1,q_1,a_2,q_2,s)}{\ph(s)\ph([q_2;s])\ph([q_1;s])} \\
    &\qquad + {\rm O} \Big(  Y\sum_{s\leq L^C} 
      \frac{\mu^2(s)}{\ph([q_2;s])} \;
     \sideset{}{^{\displaystyle *}}\max_{f([q_1;s])} 
    |\Delta(X_1,X_1+Y; [q_1;s], f)| \Big),
\end{align*}
with
\[
   G(n;a_1,q_1,a_2,q_2,s):=  \sideset{}{^{\displaystyle *}} \sum_{b\:(s)}
      \sideset{}{^{\displaystyle *}} \sum_{\substack{c\:(s)\\ c\equiv a_2 ((q_2,s)) }} \:
    \sideset{}{^{\displaystyle *}} \sum_{\substack{d\:(s)\\ d\equiv a_1 ((q_1,s)) }}
          e\,\Big( -(n-d-c)\frac{b}{s} \Big).
\]

The ${\rm O}$-term gives an admissible error due to Theorem \ref{pepisa}
(the result of Perelli, Pintz and Salerno \cite{PePiSa}). For this,
the assumptions $Q_{1}\ll YX_{1}^{-1/2}L^{-B}$ and
$X_{1}^{3/5+\varepsilon_{2}}\ll Y$ are used, cp.~also \cite[Sec.~2.1]{HalAMH}.

In the other term, the partial sum with $s\leq L^C$ can be replaced by the full
singular series, giving an admissible error. This can also be proved as
in \cite[Sec.~2.2]{HalAMH}, where we just have to set $q_3=1$.
In addition, there the singular series is obtained in Euler product form
being the term given here:

It equals $0$ if
$(q_1,q_2,n-(a_1+a_2))>1$, and in the notation of \cite{HalAMH}, 
it can be given in its Euler product form as
\[
   \frac{1}{2\ph(q_1)\ph(q_2)} \prod_{p,(A)\text{ or } (D)} 
   \Big(1-\frac{1}{(p-1)^2}\Big)
   \cdot \prod_{p,(B)} \Big(1+\frac{1}{(p-1)^3} \Big) \cdot
   \prod_{p,(C) \text{ or } (F)}  \frac{p}{p-1},
\]
where
\begin{align*}
  (A) &\Leftrightarrow p\mid n, p\nmid q_1, p\nmid q_2 
   \qquad (B) \Leftrightarrow p\nmid n, p\nmid q_1, p\nmid q_2 \\
  (C) &\Leftrightarrow p\mid n-a_1, p\mid q_1,
p\nmid q_2 \text{ or } p\mid n-a_2, p\mid q_2, p\nmid q_1 \\
   (D) &\Leftrightarrow p\nmid n-a_1,p\mid q_1,p\nmid q_2 
     \text{ or } p\nmid n-a_2,p\mid q_2,p\nmid q_1 \\
    (F) &\Leftrightarrow p\mid q_1, p\mid q_2, p\nmid n-(a_1+a_2).
\end{align*}

So we are done with Theorem \ref{th:1}.
\qed

\section{Proof of Theorem \ref{th:2}--Perelli/Pintz-approach}
\label{seceinAP}

We show: 

\begin{theorem}
\label{th:binGBoneprime}
Let $X_2\geq Y\gg X_{2}^{7/12+\varepsilon_{1}}$, 
let $Y^{1/3+\eps_{2}}\ll R \leq Y$, let $Q\ll R^{3/2}X_{1}^{-1/2}$
and $Y\asymp X_1-X_2-Y\geq 0$. Then, for any fixed positive
integer $a$ with $a\leq X_{1}$, we have
\[
  \sum_{q\leq Q}
    \sum_{\substack{k_1 \in [X_1,X_1+R]\\ 2k_1\equiv a\;(q)}} 
    \Big| \sum_{\substack{p_2+p_3=2k_1 \\ p_2 \in [X_{2},X_2+Y]}}
    \log p_2 \log p_3 -\mathfrak{S}(2k_1)Y   \Big| 
    \ll RY L^{-A},
\]
whereas for $Q\ll R^{1/2}$ instead of $Q\ll R^{3/2}X_{1}^{-1/2}$,
the estimate holds true with 
$\max_{a\:(q)}$ inserted after the sum over $q$.
\end{theorem}

Proof:
We start with the following estimate from
Theorem \ref{th:GBshortIV}.

By Lemma \ref{lHM1}, for fixed $a\leq X_1$
and since $Q^{2/3}X_1^{1/3}\ll R$, we have

\begin{align}
\label{th3var}
   &\sum_{q\leq Q} \sum_{\substack{k_1 \in [X_1,X_1+R] \\ 2k_1\equiv a\:(q) }}
   \Big| \sum_{\substack{ p_2+p_3=2k_1 \\ p_2 \in [X_2,X_2+Y]}} \log p_2 \log p_3 
    - \mathfrak{S}(2k_1)Y \Big| \notag \\ 
   \ll &R^{1/2} L^{3/2} \Big(\sum_{k_1 \in [X_1,X_1+R]} 
   \Big| \sum_{\substack{ p_2+p_3=2k_1 \\ p_2 \in [X_2,X_2+Y]}} \log p_2 \log p_3 
    - \mathfrak{S}(2k_1)Y \Big|^2 \Big)^{1/2}.
\end{align}

By Theorem \ref{th:GBshortIV} above, the expression in brackets is
$\ll R Y^{2} L^{-A-3}$, and we get the desired estimate 
$RY L^{-A}$ for the left hand side in Theorem \ref{th:binGBoneprime}. 

Note that for $Q\ll R^{1/2}$, we can apply
Lemma \ref{lHM2}  in order to get
an estimate being uniform for all residues $a$, namely 
in the same way we get 
\begin{equation*}
   \sum_{q\leq Q} 
    \max_{a\:(q)}
    \sum_{\substack{k_1 \in [X_1,X_1+R] \\ 2k_1\equiv a\:(q) }}
    \Big| \sum_{\substack{ p_2+p_3=2k_1 \\ p_2 \in [X_2,X_2+Y]}} \log p_2 \log p_3 
    - \mathfrak{S}(2k_1)Y \Big| \ll  RY L^{-A}.
\end{equation*}
\qed

Proof of Theorem \ref{th:2}:
Let us write $Y_{1}$ for $R$ and $Y_{2}$ for $Y$.
The proof now follows the same idea as in the proof of Theorem \ref{th:1}:
Use estimate \eqref{th3var} with $2k_1$ of the form
$2k_1=n-p_1$ and multiply with the weight $\log p_1$.
The necessary conditions for using \eqref{th3var} have been formulated
in Theorem \ref{th:2}.

But we use \eqref{th3var} with the singular series $\mathfrak{S}(2k_1)$
in the main term replaced by its partial sum for $s\leq L^C$ 
(cp. \eqref{eq:sing4}).  
This is possible, the contribution of the series for 
$s>L^{C}$ gives an admissible error, as shown in \cite[p.~45]{PePi}.

So we deduce the estimate
 \begin{gather}
   \label{eq:h23}
  \sum_{q\leq Q} 
  \Big| \sum_{\substack{ p_1+p_2+p_3=n \\ p_i \in [X_i,X_i+Y_i], i=1,2
 \\ p_1\equiv a\:(q)}}\log p_1 \log p_2 \log p_3  \notag \\
   - Y_{2} \sum_{\substack{ p_1\in [X_{1},X_1+Y_{1}]\\ p_1\equiv a\:(q) }}
     \log p_1 \sum_{s\leq L^C} \frac{\mu^2(s)}{\ph^2(s)} 
      \sideset{}{^{\displaystyle *}}\sum_{b\:(s)}\,e\Big(-(n-p_1)\frac{b}{s}\Big) 
     \Big| \ll Y_{1} Y_{2} L^{-A}.
 \end{gather}
Here, the main term is
\begin{align*}
    &Y_{2} \sum_{\substack{ p_1\in [X_{1}, X_1+Y_{1}]\\ p_1\equiv a\:(q) }}
     \log p_1 \sum_{s\leq L^C} \frac{\mu^2(s)}{\ph^2(s)} 
      \sideset{}{^{\displaystyle *}}\sum_{b\:(s)} 
       \sideset{}{^{\displaystyle *}}
      \sum_{\substack{c\:(s)\\c\equiv a\:((q,s))\\ c\equiv p_1 \:(s)}}
       \,e\Big(-(n-c)\frac{b}{s}\Big) \\ 
    &= Y_{2} \sum_{s\leq L^C} \frac{\mu^2(s)}{\ph^2(s)} 
      \sideset{}{^{\displaystyle *}}\sum_{b\:(s)} 
       \sideset{}{^{\displaystyle *}}\sum_{\substack{c\:(s)\\c\equiv a\:((q,s))}}
       \sum_{\substack{p_1\in [X_{1}, X_1+Y_{1}]\\ p_1\equiv a\:(q) \\
           p_1\equiv c\: (s) }} \log p_1 \:e\Big(-(n-c)\frac{b}{s}\Big) \\
    &= \sum_{s\leq L^C} \frac{\mu^2(s)}{\ph^2(s)} F(n;a,q,s)
       \frac{Y_{1}Y_{2}}{\ph([q;s])} \\
       &\qquad + {\rm O} \Big(  Y_{2} \sum_{s\leq L^C}
           \frac{\mu^2(s)}{\ph(s)} \; 
       \sideset{}{^{\displaystyle *}}\max_{f \,([q;s])}| 
      \Delta(X_1,X_1+Y_{1};[q;s],f)|\Big),
\end{align*}
where \[F(n;a,q,s):=\sideset{}{^{\displaystyle *}}\sum_{b\:(s)} 
  \sideset{}{^{\displaystyle *}}\sum_{\substack{c\:(s)\\c\equiv a\:((q,s))}}
   \:e\Big(-(n-c)\frac{b}{s}\Big) \]
and where we used $(q,a)=1$ in the last step, what we can assume
w.l.o.g., else \eqref{eq:h23} holds true clearly. 
Note that the version of \eqref{eq:h23} can be shown with 
$\max_{a\:(q)}$ inserted after the sum over $q$ 
by the use of the supplement of Theorem \ref{th:binGBoneprime}
coming from Lemma \ref{lHM2}.

The ${\rm O}$-term is $\ll Y_{1}Y_{2} L^{-A}$ and therefore admissible 
since we may apply Theorem \ref{pepisa}
for $Y_{1}\gg X_1^{3/5+\eps}$ and $Q\ll Y_{1}X_{1}^{-1/2}L^{-B}$.

The error that comes now from the replacement of the partial sum
$\sum_{s\leq L^C}$ by the  
full singular series can be estimated to be admissible in exactly the
same way as in 
\cite{HalAMH}, Section 2.2: there, let $q_1=q,a_1=a, q_2=q_3=1$ and the same 
proof applies here, too. And since the singular series sums up to the given one
$\mathfrak{T}(n,q,a)$, we are done.

Further, the same proof with Lemma \ref{lHM2} instead of Lemma \ref{lHM1}
 gives the supplement of the theorem, 
and $Q\ll Y_{1}^{1/2}$ is needed then as assumption.
\qed


\section{Proof of the Corollaries
 with sieve methods}
\label{secdrei}

For the proof of these Corollaries we proceed as in the proof
of Meng \cite{Meng}.
This works in exactly the same way for Corollaries \ref{cor:2} 
and \ref{cor:3}.
We give an indication of these proofs now. They rely on the
application of Theorem 9.3 in \cite{HR}, there 
$\Lambda_{s}:=s+1-\frac{\log (4/(1+3^{-s}))}{\log 3}$ with a natural
number $s$, and it is assumed that we sieve for a
finite set $\mathcal{A}$, where $\xi$ serves as an approximation
for $\#\mathcal{A}$ (this corresponds $X$ in \cite{HR}).

\zeile
\textbf{Proof of Corollary \ref{cor:2}.}

Here we work with the sequence $\mathcal{A}$ of all $p_2+2$
with $k_1= p_2+p_3$ such that
$p_2$ lies in the short interval $p_2 \in [X_2,X_2+Y]$ with $Y=X_2^{\theta}$,
and $k_1$ lies in the short interval $k_1\in [X_{1}, X_1+Y]$
with $Y\gg X_1^{2/3+\eps}$. We have $\xi\gg YL^{-2}$.
We use then Theorem 9.3 in \cite{HR}, its conditions 
can be checked in the same way as in
\cite{Meng}, where for (d) we have to apply Theorem \ref{th:binGBtwoprimes}
in the version \eqref{eq:Th4var} below
with $Q_1=1$ instead. This works with $\alpha=1-1/2\theta$, and 
$|a|\leq \xi^{\alpha (\Lambda_s-\delta)}$ (this is Condition (9.3.6) in
\cite{HR}) holds if 
$1/\theta < (1-1/2\theta)\Lambda_s$.  
This is true for $s=3$ since $\Lambda_{3}\geq 2.771$,
and $\theta\geq 0.861$. \qed


\zeile
\textbf{Proof of Corollary \ref{cor:4}.}

Consider $X_{i},Y_{i}$, $i=1,2$, and $n$ as given.
Now we work with the sequence $\mathcal{A}$ of all $p_2+2$
with $n= p_{1}+ p_2+p_3$ such that
$p_i$ lies in the short interval $p_i\in [X_{i},X_i+Y_{i}]$, $i=1,2$.
We have $\xi\gg Y_{1}L^{-3}$. Write $Y_{1}=Y_{2}^{\eta}$ with
$\frac{1}{3}<\eta\leq 1$.
Again Theorem 9.3 in \cite{HR} applies with 
$\alpha=\min(1-\frac{1}{2\theta_{1}},\frac{3}{2}-\frac{1}{2\eta})$,
this time with Theorem \ref{th:2}.
Now since $|a|\leq X_{1}+Y_{1}+2\ll X_{1}\ll Y_{1}^{1/\theta_{1}}$,
the condition $|a|\leq \xi^{\alpha (\Lambda_s-\delta)}$ is true if
$1<\theta_{1}\alpha \Lambda_{s}$. 

First case: If
$1-\frac{1}{2\theta_{1}}<\frac{3}{2}-\frac{1}{2\eta}$,
we have as in Corollary \ref{cor:2} before $s=3$, $\theta_{1}\geq
0.861$, then we have to choose 
$\eta>(1+1/\theta_{1})^{-1}\geq 0.463$.
Second case: Here $\eta$ has to be $\eta\leq 0.5$, so that $\theta_{1}\leq
1$ can be such that $1/\theta_{1} \leq 1/\eta-1$. 
Also $\alpha \Lambda_{s}=(3/2-1/2\eta) \Lambda_{s}>1$, what gives 
$\eta>0.439$ for $s=3$. So with $\eta$ in this range, one can choose
$\theta_{1}$ such that
$1/\theta_{1}<\min(1/\eta-1,(3/2-1/2\eta)\Lambda_{3})$;
we will have then that $\theta_{1}>0.782$.
\qed

\zeile
Now Corollary \ref{cor:3} is a consequence of Corollary \ref{cor:2}:

\zeile
\textbf{Proof of Corollary \ref{cor:3}.}

Consider the number of $n-p_1$ such that $p_1+2=P_2$ and $p_1\in [X_{1},X_1+Y]$.
By Wu's Theorem in \cite{Wu}, 
we know that this number is $\gg Y L^{-2}$ if $Y=X_1^{\theta_1}$ 
for $\theta_1\geq  0.971$.
So not all of them can be exceptions in Corollary \ref{cor:2},
so there is at least one of them being the sum of two primes $p_2$ and
$p_3$, where 
$p_2$ lies in a short interval of length $Y=X_2^{\theta_2}$ with
$\theta_2\geq 0.861$, 
such that $p_2+2=P_3$. Note that Corollary \ref{cor:2} is applicable
since we assumed that $X_1+X_2+2Y\leq n\ll Y^{3/2-\eps}$, so for the lower
bound $n-X_1-Y$ of the numbers $n-p_1$
we have $X_2+Y\leq n-X_1-Y \ll Y^{3/2-\eps}$.
\qed


\zeile\zeile
Now for Corollary~\ref{cor:1} we have to work
in a slightly different style; therefore we here give
the proof of Corollary~\ref{cor:1} in detail.
As sieve method we need Theorem 10.3 of Halberstam 
and Richert \cite{HR}, which we present first: For this
let $\mathcal{A}$  be a finite sequence of integers, $\mathcal{P}$
an infinite set of primes, and  
$\mathcal{A}_{d}$ the sequence of all $a\in\mathcal{A}$ with $d\mid a$.
Further, for the number of elements in $\mathcal{A}_d$ we write
\[ \#\mathcal{A}_{d}=\frac{\omega(d)}{d}\xi + R_{d}  \]
with a multiplicative arithmetic function $\omega$ such that
$\omega(p)=0$ for $p\not\in\mathcal{P}$. Let $L:=\log \xi$,
$\xi\geq 2$. 
We assume that $(a,p)=1$ for any prime $p\not\in\mathcal{P}$ and
any $a\in\mathcal{A}$.

\begin{theorem}
\label{sth} (Theorem 10.3 of \cite{HR})
Assume that 
\begin{enumerate}[(a)]
\item there exist a constant $A_{1}>0$ such that
    \[
         1\leq \frac{1}{1-\frac{\omega(p)}{p}} \leq A_{1}
     \]
for all $p\in\mathcal{P}$,
\item for a constant $\kappa>1$ (the sieve dimension), a constant $A_{2}\geq 1$ 
 and for all real $v,w$ with $2\leq v \leq w$ we have
     \[
        \sum_{\substack{v\leq p\leq w \\
              p\in\mathcal{P}}}  \frac{\omega(p)}{p} \log p 
            \leq \kappa \log \frac{w}{v} + A_{2},
    \]
\item for a constant $A_{3}>0$ and for all real $z,y$ with
$2\leq z\leq y\leq \xi$ we have
    \[
         \sum_{\substack{ z\leq p<y \\ p\in\mathcal{P} }}
         \#\mathcal{A}_{p^{2}} \leq A_{3} \Big(\frac{\xi L}{z}+y
         \Big),
    \] 
(Any fixed power of $L$ is here possible, too, as remarked
by Halberstam and Richert \cite{HR}),
\item for constants $0<\alpha<1$ and $A_{4},A_{5}>0$ we have
   \[
      \sum_{d<\frac{\xi^{\alpha}}{L^{A_{4}}}} \mu^{2}(d) 3^{\nu(d)}
      |R_{d}| \leq A_{5} \frac{\xi}{L^{\kappa+1}}.
  \]  
\end{enumerate}
Further assume that there exists a real $\mu>0$ such that $|a|\leq
\xi^{\alpha \mu}$ for all $a\in\mathcal{A}$. Let $\zeta\in\R$,
$0<\zeta<\nu_{\kappa}$ for a certain real $\nu_{\kappa}>1$ depending on
$\kappa$ only, and let $r\in\N$ with
\begin{equation}
\label{eq:label5}
    r> (1+\zeta)\mu-1+(\kappa+\zeta)\log\frac{\nu_{\kappa}}{\zeta}
   - \kappa - \zeta\frac{\mu-\kappa}{\nu_{\kappa}}.
\end{equation}
Then there exists a $\delta=\delta(r,\mu,\kappa,\zeta)>0$ with
\begin{equation}
\label{eq:sthestimate}
    |\{ P_{r};\, P_{r}\in \mathcal{A} \}| \geq \delta
   \frac{\xi}{(\log \xi)^{\kappa}}\Big(1-\frac{C}{\sqrt{\log X}}\Big)
\end{equation}
for a constant $C>0$ depending at most on $r,\mu,\zeta$
(as well as on the $A_{i}$'s, $\kappa$ and $\alpha$).

\end{theorem}

\zeile
We use this theorem in the case $\kappa=2$, 
for which the numerical value $\nu_{\kappa}=4.42\dots$ is known
(see \cite[(7.4.9)]{HR}).

We are going to apply Theorem \ref{th:binGBtwoprimes} and Theorem
\ref{th:1}, 
but we need them in non-weighted version. With an obvious partial
summation, we can 
transform the first estimate in Theorem \ref{th:binGBtwoprimes} into
\begin{align}
\label{eq:Th4var}
  \sum_{\substack{q_1\leq Q_1 \\q_2\leq Q_2}}
   \sum_{\substack{k_1 \in [X_1,X_1+R]\\ 2k_1\equiv a_1 \:(q_1)}}
   \Big|  \sum_{\substack{p_2+p_3=2k_1 \\ p_2 \equiv a_2 \:(q_2) \\
       p_2 \in [X_2,X_2+Y]}} 1 - \mathfrak{S}(2k_1,q_2,a_2) 
     \int_{X_2}^{X_2+Y} \frac{dt}{\log t \log (k_1-t) } \Big|\notag \\
     \ll Y^2 L^{-A}. 
\end{align}

A non-weighted version of the estimate in Theorem
\ref{th:1} is the following: 
\begin{align}
\label{eq:Cor4var}
   \sum_{\substack{q_1\leq Q_1\\ q_{2}\leq Q_{2}}} \Big|
       \sum_{\substack{p_1+p_2+p_3=n  \\ p_i \in [X_i,X_i+Y_i], i=1,2 \\
           p_i \equiv a_i \:(q_i), i=1,2 }} 1 
      - \mathfrak{T}(n,q_1,a_1,q_2,a_2)H(X_1,X_2,Y,n) \Big| \notag\\
   \ll Y^2 L^{-A},
\end{align}
where 
\begin{equation}
\label{eq:Ha}
   H(X_1,X_2,Y,n):= \int_{X_1}^{X_1+Y} \frac{1}{\log v} 
          \int_{X_2}^{X_2+Y} \frac{dt}{\log t \log (n-v-t)} dv.
\end{equation}
This version can be obtained in the same way as the original version
of the Theorem,
but where estimate \eqref{eq:Th4var} and the $\pi$-version of the Theorem
of Perelli, Pintz and Salerno \cite{PePiSa} is used, namely
\[ 
   \sum_{q\leq Q}  \sideset{}{^{\displaystyle *}} \max_{a\:(q)}
    \max_{h\leq Y} \Big|\sum_{\substack{p\in[X,X+h]\\p\equiv a\:(q)}} 1
    - \frac{1}{\varphi(q)} \int_X^{X+h} \frac{dt}{\log t} \Big| \ll YL^{-A}.
\]
This version can be gained from the original one in the same way like 
Bombieri-Vinogradov's theorem can be transformed from a $\psi$-version 
into a $\pi$-version, as done in \cite{Brue}.

\zeile
\textbf{Proof of Corollary \ref{cor:1}.} 

To prepare the application of Theorem \ref{sth}, we set
$\mathcal{P}:=\{ p;\, p\neq 2 \}$ and for large odd $n\not\equiv 1\:(6)$, 
where $n\geq X_1+X_2+2Y$. Let $\mathcal{A}$
denote the sequence of all $(p_1+2)(p_2+2)$
such that $p_1+p_2+p_3=n$  and $p_i \in [X_i,X_i+Y]$ for $i=1,2$ is solvable.
So this sequence assigns the number $(p_1+2)(p_2+2)$ to each pair
$(p_1,p_2)$, and these pairs may be given in some order.

Then $(a,2)=1$ holds for all
$a\in\mathcal{A}$. Let $\mathcal{A}_d$ denote the sequence of all elements 
$a\in \mathcal{A}$ with $d\mid a$. 

Now
\begin{align*}
   |\mathcal{A}_d| &= 
\sum_{\substack{p_i \in [X_i,X_i+Y], i=1,2\\ p_1+p_2+p_3=n\\
 (p_1+2)(p_2+2)\equiv 0\: (d) }}
    1  = \sum_{t\mid d}   \sum_{\substack{p_i \in [X_i,X_i+Y], i=1,2\\ 
        p_1+p_2+p_3=n\\ p_1\equiv -2 \:(t) \\p_2\equiv -2\: (d/t)\\
        ((p_{1}+2)/t,d/t)=1 }} 1\\
       & =\sum_{t\mid d}  \sum_{\substack{p_i \in [X_i,X_i+Y], i=1,2\\ 
        p_1+p_2+p_3=n\\ p_1\equiv -2 \:(t) \\p_2\equiv -2\: (d/t)}}\;
        \sum_{\substack{s\mid (p_{1}+2)/t,\\ s\mid d/t}} \mu(s) 
       = \sum_{t\mid d} \sum_{s\mid d/t} \mu(s) \sum_{\substack{p_i
           \in [X_i,X_i+Y], i=1,2\\  
          p_1+p_2+p_3=n\\ p_1\equiv -2 \:(st) \\p_2\equiv -2\: (d/t)}} 1,
\end{align*}
which we want to approximate by
\[
\sum_{t\mid d} \sum_{s\mid d/t} \mu(s) 
    \mathfrak{T}(n,st,-2,d/t,-2) H(X_1,X_2,Y,n).
\]
So in view of Theorem \ref{sth},
we write $|\mathcal{A}_d|= \frac{\omega(d)}{d}\xi + R_d$ with
\[
   \xi:=\frac{1}{2} H(X_1,X_2,Y,n) \prod_{p\mid n} \Big(1-\frac{1}{(p-1)^2}\Big)
     \prod_{p\nmid n} \Big(1+\frac{1}{(p-1)^3}\Big)
\]
and
\begin{align}
\label{eq:omegad}
   \omega(d):= d  \sum_{t\mid d} \sum_{s\mid d/t} &\frac{\mu(s)}{\ph(st)\ph(d/t)}
   \prod_{\substack{p\nmid n+2,p\mid st, p\nmid d/t \\ 
   \text{or } p\nmid n+2,p\nmid st, p\mid d/t } }
   \Big(1-\frac{1}{(p-1)^2}\Big) \prod_{p\mid n,p\mid ds} 
  \Big(1-\frac{1}{(p-1)^2}\Big)^{-1} 
  \notag   \\ &\cdot\prod_{p\nmid n,p \mid ds} \Big(1+\frac{1}{(p-1)^3}\Big)^{-1} 
   \prod_{\substack{p\mid n+2, p\mid st, p\nmid d/t \\
     \text{or } p\mid n+2,p\mid d/t, p\nmid st \\ 
     \text{or }  p \nmid n+4, p\mid st,p\mid d/t }}
    \Big(1+\frac{1}{p-1}\Big).
\end{align}

From this formula it can be seen that $\omega$ is multiplicative in $d$.

Now we are going to check all conditions of Theorem \ref{sth}
in this setting.

By a computation, from \eqref{eq:omegad} we can deduce that, 
for a prime $\ell\neq 2$ we have:
\[
    \omega(\ell)=
    \begin{cases}
      \frac{2\ell}{\ell-2}-\frac{1}{\ell-1},
   &\text{if } \ell\mid n,\ \ell\nmid n+2,\ \ell \nmid n+4,\\[1ex]
      \frac{\ell^{2}(2\ell-3)}{(\ell-1)^{3}+1}, 
      &\text{if } \ell \nmid n,\ \ell\mid n+2,\ \ell\nmid n+4,\\[1ex]
      \ell\frac{2(\ell-1)^{2}-1-\ell}{(\ell-1)^3+1}, &\text{if } \ell\nmid n,\ 
     \ell\nmid n+2,\ \ell \mid n+4,\\[1ex]
      \ell\frac{2(\ell-1)^{2}-2-\ell}{(\ell-1)^{3}+1},
      &\text{if } \ell\nmid n,\ 
     \ell\nmid n+2,\ \ell\nmid n+4.
    \end{cases}
\]

So it follows that (a) holds (note that the second case, where
$\omega(3)/3=1$, does not occur for $\ell=3$ since $n\not\equiv 1 \:(6)$).

\zeile
For condition (b) we see that for all $w\geq v\geq 2$, we have
\[
  \sum_{v\leq p\leq w} \frac{\omega(p)}{p} \log p
    \leq 2\sum_{v\leq p\leq w} \frac{\log p}{p-2} 
   \leq 2\log\frac{w}{v} +  A_2
\]
for a $A_2\geq 1$ since
\[
   \sum_{p\leq x} \frac{\log p}{p-2}
    = \sum_{p\leq x} \frac{\log p}{p} + 2 \sum_{p\leq x} \frac{\log p}{(p-2)p} 
    = \log x + {\rm O}(1).
\]
So $\kappa=2$ works as sieve dimension.

\zeile
For condition (c) we see that
$|\mathcal{A}_{p^2}|\ll \frac{Y^2}{p^2}+1$, so

\[
  \sum_{z\leq p<y} |\mathcal{A}_{p^{2}}| \ll \sum_{z\leq q<y}
  \Big(\frac{Y^2}{q^{2}} +1\Big) \ll \frac{Y^2}{z}\log y +y \ll
  \frac{\xi}{z}L^{4} + y.
\]

\zeile
We are going to check now condition (d). Let
$D=Y^{1-1/2\theta}L^{-B}=\xi^\alpha L^{-B}$,
where $\alpha=\frac{1}{2}-\frac{1}{4\theta}$
and $\theta=\min(\theta_1,\theta_2)$. 
Abbreviate
\[
A(q_{1},q_{2}):=
   \sum_{\substack{p_i
          \in [X_i,X_i+Y], i=1,2\\  
          p_1+p_2+p_3=n\\ p_1\equiv -2 \:(q_{1}) \\p_2\equiv -2\:
          (q_{2})}} 1 -\mathfrak{T}(n,q_{1},-2,q_{2},-2)H(X_{1},X_{2},Y,n).
\]  
Then we have
\begin{align*}
  \sum_{d\sim D} &\mu^{2}(d) 3^{\nu(d)} |R_{d}| 
   =  \sum_{\substack{d\sim D\\(d,2)=1}} \mu^2(d) 3^{\nu(d)}
       \Big| \sum_{t\mid d} \sum_{s\mid d/t} \mu(s)A(st,d/t)\Big|\\
    &\ll  \sum_{\substack{d\sim D\\(d,2)=1}} \mu^2(d) 3^{\nu(d)}
        \sum_{t\mid d} \sum_{s\mid d/t}  \frac{Y}{(sd)^{1/2}}
        |A(st,d/t)|^{1/2}\\
   &\ll Y\sum_{q_{1},q_{2}\leq 2D} \Big(\sum_{\substack{t\mid q_{1}\\
       (t,q_{2})=1}} \mu^{2}(tq_{2}) 3^{\nu(tq_{2})} q_{1}^{-1/2}
   q_{2}^{-1/2} \Big)
   |A(q_{1},q_{2})|^{1/2}\\
   &\ll Y \Big(\sum_{q_{1},q_{2}\leq 2D} \frac{1}{q_{1}q_{2}}\sum_{\substack{t\mid
     q_{1}\\ (t,q_{2})=1}} \mu^{2}(tq_{2})3^{2\nu(tq_{2})} \Big)^{1/2}
     \Big(\sum_{q_{1},q_{2}\leq 2D} |A(q_{1},q_{2})|\Big)^{1/2} \\
   &\ll Y^{2} L^{-A-3}\ll \xi L^{-A},
\end{align*}
using Theorem \ref{th:1} in the version \eqref{eq:Cor4var},
and since
\[
   \sum_{q_{1},q_{2}} \frac{1}{q_{1}q_{2}} \sum_{t\mid q_{1}}
   3^{2\nu(t)}3^{2\nu(q_{2})} \ll
   \sum_{q_{1}}\frac{\tau(q_{1})^{5}}{q_{1}} \sum_{q_{2}}
   \frac{\tau(q_{2})^{4}}{q_{2}} \ll L^{48}.
\]

Now we need to know whether $|a|\leq\xi^{\alpha \mu}$
for all $a\in\mathcal{A}$ and some $r\geq 2$.
We have $a=(p_1+2)(p_2+2)\leq (X_1+Y+2)(X_2+Y+2)\ll Y^{2/\theta}$,
so we have to take $\mu>4/(2\theta-1)>4$. 
Write $\mu=4+\Delta$, then the right hand side of \eqref{eq:label5} is $<9$ for
$\zeta=0.360$ and $\Delta=0.628$.
This gives the value $\theta\geq 0.933$.

$r=9$ is the smallest possible value, and also $\theta\geq 0.933$
is optimal such that $r\geq 9$ can be chosen.

Therefore Theorem \ref{sth} applies with $r=9$. From this we deduce
that there always exists a $P_{9}$ in $\mathcal{A}$, their number is
at least $\gg \xi/(\log\xi)^{2}$.  So we are
done with the proof of Corollary \ref{cor:1}.
\qed

\subsection*{Acknowledgments}
The author would like to thank the referee
for a careful reading and many valuable comments.
%

\end{document}